\documentclass[final,3p,times]{elsarticle}

\usepackage{bbm}
\usepackage[T1]{fontenc}
\usepackage[utf8]{inputenc}
\usepackage{lmodern}
\usepackage{microtype}
\usepackage{amsmath,amssymb,mathtools,amsthm}
\usepackage{enumitem}
\usepackage{xcolor}
\usepackage[colorlinks=true,linkcolor=blue!45!black,citecolor=blue!45!black,urlcolor=blue!45!black]{hyperref}

\biboptions{sort&compress}

\newtheorem{theorem}{Theorem}[section]
\newtheorem{proposition}[theorem]{Proposition}
\newtheorem{lemma}[theorem]{Lemma}
\newtheorem{corollary}[theorem]{Corollary}
\theoremstyle{definition}
\newtheorem{definition}[theorem]{Definition}

\newcommand{\R}{\mathbb R}
\newcommand{\N}{\mathbb N}
\newcommand{\id}{\operatorname{id}}
\newcommand{\Int}{\operatorname{int}}
\newcommand{\cone}{\operatorname{cone}}
\newcommand{\conv}{\operatorname{conv}}
\newcommand{\supp}{\operatorname{supp}}
\newcommand{\End}{\operatorname{End}}
\newcommand{\Disc}{\mathsf{Disk}}
\newcommand{\RC}{\mathcal R_C}
\newcommand{\dd}{\,d}
\newcommand{\one}{\mathbf 1}
\newcommand{\tr}{\operatorname{tr}}
\newcommand{\diag}{\operatorname{diag}}

\newcommand{\graph}{\operatorname{graph}}

\newcommand{\bx}{\mathbf{x}}
\newcommand{\by}{\mathbf{y}}
\newcommand{\bz}{\mathbf{z}}
\newcommand{\bw}{\mathbf{w}}
\newcommand{\bu}{\mathbf{u}}
\newcommand{\bv}{\mathbf{v}}
\newcommand{\bh}{\mathbf{h}}
\newcommand{\be}{\mathbf{e}}
\newcommand{\bp}{\mathbf{p}}
\newcommand{\bq}{\mathbf{q}}
\newcommand{\ba}{\mathbf{a}}
\newcommand{\bb}{\mathbf{b}}

\newcommand{\bg}{\mathbf{g}}
\newcommand{\bs}{\mathbf{s}}
\newcommand{\blambda}{\boldsymbol{\lambda}}
\newcommand{\bmu}{\boldsymbol{\mu}}

\journal{Preprint}

\begin{document}

\begin{frontmatter}

\title{FS-domains are not always RB-domains\tnoteref{t1}}
\tnotetext[t1]{Research supported by NSF of China (Nos. 12471439, 12231007).}

\author[addr1]{Yuxu Chen}
\address[addr1]{School of Mathematics, Sichuan University, Chengdu, P.R. China  610065}
\ead{chenyuxu@scu.edu.cn}

\author[addr1]{Hui Kou}
% \address[addr1]{School of Mathematics, Sichuan University}}
\ead{kouhui@scu.edu.cn}

\author[addr1]{Zhenchao Lyu}
% \address[addr1]{School of Mathematics, Sichuan University}}
\ead{zhenchaolyu@scu.edu.cn}

\begin{abstract}
We prove that Lawson's planar closed-disk domain is not an RB-domain.  This domain is the dcpo of all closed disks in the Euclidean plane, together with the whole plane as bottom, ordered by reverse inclusion.  
Since this domain is an FS-domain, it gives a concrete example of an FS-domain which is not an RB-domain,  answering negatively the long-standing open problem in domain theory of whether FS-domains and RB-domains are identical.  
\end{abstract}


\end{frontmatter}

\section{Introduction}

Domain theory was introduced by Dana Scott to give mathematical model to denotational semantics of programming languages \cite{Scott1976,Scott1982,Scott1993}. 
In order to interpret higher order types in programming languages,  a particularly important requirement on categories of semantic domains is that they should be cartesian closed (i.e. closed under function spaces). The most important objects in domain theory are continuous domains and algebraic domains.  However, the category \textbf{CONT} (resp. \textbf{ALG} ) of continuous domains  (resp. algebraic domains) and Scott continuous maps is not cartesian closed.   Hence,  looking for cartesian closed full subcategory of \textbf{CONT} (resp. \textbf{ALG}) is an important task in domain theory. In 1983, Smyth \cite{Smyth1983} proved that the category of \textbf{SFP} (aka countably based bifinite domains), which was introduced by Gordon Plotkin \cite{Plotkin1981}, is the largest cartesian closed full subcategory of the category of $\omega$- algebraic pointed dcpos. Following Smyth's path,
 Achim Jung \cite{Jung1989,Jung1990} introduced two special domains called $L$-domain and FS-domain and  showed that, the categories \textbf{CL} and \textbf{FS} consisting  of continuous $L$-domains and  FS-domains respectively, are the only two  maximal cartesian closed full subcategories of \textbf{CONT} if one works in pointed continuous domains. Replacing FS-domais by bifinite domains, the  algebraic case  also holds. Particularly, bifinite domains are equal to algebraic FS-domains, and the retracts of  bifinite domains, called RB-domains,  form a cartesian closed category \cite[Proposition~4.2.12]{AbramskyJung1994}.  
 Since each RB-domain is an FS-domain, it arises 
a famous open problem in domain theory since 1990s  \cite{AbramskyJung1994,Gierz2003} as follows : are  FS-domains and RB-domains  identical?  In the setting of $L$-domains, Liang and Keimel gave a positive answer \cite{LiangKeimel1999} and similar results are obtained in \cite{ZouLiGuo2018}. 
 In this paper we give a negative answer to this question by proving that Lawson's planar closed-disk domain is not an RB-domain.

We denote by $\Disc$ the domain of closed disks in the Euclidean plane. Its elements are all closed disks
\[
        \overline B(\bz,r)=\{\bw\in\R^2: |\bw-\bz|\le r\},
        \qquad \bz\in\R^2,\quad r\ge 0,
\]
together with the whole plane \(\R^2\), ordered by reverse inclusion.  The whole plane is the least element.  Jung discussed this example in his classification paper and explicitly wrote that it was suggested by Jimmie Lawson \cite{Jung1990}.  
The same example is also recorded in Abramsky and Jung's domain theory chapter as a well-structured FS-domain for which the RB question was open \cite[Section~4.2.2, Example after Proposition~4.2.12]{AbramskyJung1994}.  Lawson later proved a general FS-domain theorem for formal balls over suitable metric spaces \cite{Lawson2008}. $\Disc$ is a special case of this general result, and hence is an FS-domain.

The disk \(\overline B(\bz,r)\) can also be represented by the following point of \(\R^3\).
\[
        \bx=(-r,\bz)\in H:=(-\infty,0]\times\R^2.
\]

In this paper, we prove that \(\Disc\) is not an RB-domain.  Inspired by Keimel's work  \cite{Keimel2009} on the connection between the \(\Disc\) domain and the  convex cones, the proof uses a geometric representation of the disk domain in \(\R^3\) and a convex-geometric argument involving the Lorentz cone.  
Let \(
        C=\{(t,z)\in\R\times\R^2:t\ge |z|\}
\) be the Lorentz-cone.
From the perspective of the cone order, the reverse inclusion of disks is equivalent to the Lorentz-cone order~\cite{Keimel2009} as follows:
\[
        \bx\le_C \by \quad\Longleftrightarrow\quad \by-\bx\in C.
\]
Indeed,
\[
        \overline B(\bz,r)\supseteq \overline B(\bw,s)
        \quad\Longleftrightarrow\quad
        |\bz-\bw|\le r-s
        \quad\Longleftrightarrow\quad
        (-s,\bw)-(-r,\bz)\in C.
\]

The main idea is to show that if \(\Disc\) were an RB-domain, then there would exist finite-image \(C\)-monotone maps approximating the identity map on compact subsets of \((- \infty,0)\times\R^2\).  However, this leads to a contradiction with the properties of the Lorentz cone and its associated matrix cone.

The paper is organized as follows.  Section \ref{sec:domain-prelim} recalls the domain-theoretic definitions and introduces the closed-disk domain.  Section \ref{sec:disk-properties} records the order-theoretic properties of the disk model and proves the RB local approximation lemma.  Section \ref{sec:cone-functional} collects the Lorentz-cone, convex-geometric, and matrix facts used later.  Section \ref{sec:proof} proves the finite-image monotone matrix lemma and the main non-RB theorem.

\section{Preliminaries }\label{sec:domain-prelim}

In this section, we recall the domain-theoretic definitions and introduce the planar closed-disk domain. For the convenience of readers more familiar with domain theory than with
measure theory, we record the measure-theoretic and analytic conventions
used in the proof of the main theorem.

\subsection{Domain-theoretic definitions and the planar disk domain}
We use standard terminology from domain theory; see \cite{AbramskyJung1994,Gierz2003,Goubault2013}.  A subset \(A\) of a poset is \emph{directed} if it is nonempty and every two elements of \(A\) have an upper bound in \(A\).  A \emph{dcpo} is a poset in which every directed subset has a supremum.  A map between dcpos is \emph{Scott-continuous} if it is monotone and preserves directed suprema. Let $D,E$ be two dcpos. The pointwise order between two Scott continuous maps $f,g$ from $D$ into $E$ is defined as follows: 
$$f\leq g \quad\Longleftrightarrow\quad f(x)\leq g(x), \ \forall x\in D.$$

For elements \(x,y\) in a dcpo, one writes \(x\ll y\), called $x$ is \emph{way-below} $y$, if, whenever \(A\) is directed and \(y\le \sup A\), there is \(a\in A\) such that \(x\le a\).  A dcpo is \emph{continuous} if every element is the directed supremum of the elements way-below it.

Let \(D\) be a dcpo.  A Scott-continuous map \(f:D\to D\) is \emph{finitely separated from the identity} if there is a finite set \(M\subseteq D\) such that for every \(x\in D\) there exists \(m\in M\) with
\(
        f(x)\le m\le x.
\)
\begin{definition}[FS-domain]\cite[Definition~4.2.9]{AbramskyJung1994}
A dcpo \(D\) is an \emph{FS-domain} if there a directed family \((r_i)_{i\in I}\) of  Scott-continuous maps finitely separated from the identity \(\id_D\) such that  \(
        \bigvee_{i\in I} r_i=\id_D
\) pointwise.
\end{definition}

A \emph{deflation} on a dcpo \(D\) is a Scott-continuous finite-image map \(r:D\to D\) such that \(r\le \id_D\). A dcpo \(D\) is a \emph{bifinite domain} if the identity map \(\id_D\) is the pointwise supremum of a directed family of Scott-continuous idempotent deflations. A dcpo is an \emph{RB-domain} if it is a Scott-continuous retract of a bifinite domain.  We shall use the following standard result from the finite-deflation characterization of \emph{RB} domains.

\begin{definition}[RB-domain]\cite[Definition~9]{JungTix1998}
        A dcpo \(D\) is an \emph{RB-domain} if there is a directed family \((r_i)_{i\in I}\) of deflations on \(D\) such that \(
        \bigvee_{i\in I} r_i=\id_D
\)
pointwise. 
\end{definition}

The categories \textbf{FS}, \textbf{RB} and \textbf{Bif} consisting of all FS-domians, RB-domains and bifinite domains respectively are cartesian closed \cite{AbramskyJung1994,Jung1989,Jung1990}, and they  have the special importance in domain theory. 

Let \(\Disc\) denote the planar closed-disk domain.  Its non-bottom elements are the closed disks \(\overline B(\bz,r)=\{\bw\in\R^2: |\bw-\bz|\le r\}\) with \(\bz\in\R^2\) and \(r\ge0\).  The order is reverse inclusion:
\[
        \overline B(\bz,r)\le \overline B(\bw,s)
        \quad\Longleftrightarrow\quad
        \overline B(\bz,r)\supseteq \overline B(\bw,s).
\]
The least element is the whole plane \(\R^2\).  Equivalently, the non-bottom part may be written as the formal-ball space \(\R^2\times\R_+\), ordered by
\[
        (\bz,r)\le (\bw,s)
        \quad\Longleftrightarrow\quad
        |\bz-\bw|\le r-s.
\]

The following fact is due to Jung, with the example attributed there to Jimmie Lawson; see \cite{Jung1990}.  It also appears in the exposition of Abramsky and Jung \cite[Section~4.2.2, Example after Proposition~4.2.12]{AbramskyJung1994}; Lawson's  note \cite{Lawson2008} gives a general metric formal-ball FS result that includes the Euclidean disk case.
\begin{proposition}\cite{AbramskyJung1994,Lawson2008}\label{prop:jung-fs}
The planar closed-disk domain \(\Disc\) is an FS-domain.
\end{proposition}

\subsection{Measure-theoretic and analytic conventions}

We recall the elementary measure-theoretic and analytic facts used below~\cite{Rockafellar1970,Folland1999,EvansGariepy2015}. In this paper, integration is always understood in the Lebesgue sense. 

The Lebesgue measure of a set \(A\subseteq\mathbb R^k\) is denoted by \(\mathcal L^k(A)\).  A function \(f:\mathbb R^k\to\mathbb R\) is \emph{Lebesgue measurable} if the preimage of every Borel set is Lebesgue measurable.  A function \(f:\mathbb R^k\to\mathbb R\) is \emph{integrable} if it is Lebesgue measurable and
\(
        \int_{\mathbb R^k}|f|\,d\mathcal L^k<\infty.\)
For vector-valued or matrix-valued functions, measurability and
integrability are understood componentwise.  Thus, if
\(F=(F_{ij})\) is matrix-valued, then
\[
        \int F
        =
        \left(\int F_{ij}\right)_{ij},
\]
provided all entries \(F_{ij}\) are integrable.

For an open set \(U\subseteq\mathbb R^n\), the notation \(K\Subset U\)
means that \(K\) is compact and \(K\subseteq U\).  If a function \(g\) on
\(U\) satisfies \(\operatorname{supp}g\Subset U\), then \(g\) vanishes in a
neighbourhood of \(\partial U\).

\begin{proposition}[Tonelli--Fubini theorem]\cite[Theorem~2.37]{Folland1999}\label{prop:tonelli-fubini}
Let \(U\subseteq\mathbb R^n\) be Lebesgue measurable.

\begin{enumerate}
\item If \(G:\mathbb R\times U\to[0,\infty]\) is measurable, then
\[
        \int_{\mathbb R\times U}G(t,\bz)\,dt\,d\bz
        =
        \int_U
        \left(
        \int_{\mathbb R}G(t,\bz)\,dt
        \right)
        d\bz ,
\]
where both sides are allowed to be \(+\infty\).

\item If \(G\in L^1(\mathbb R\times U)\), then
\[
        \int_{\mathbb R\times U}G(t,\bz)\,dt\,d\bz
        =
        \int_U
        \left(
        \int_{\mathbb R}G(t,\bz)\,dt
        \right)
        d\bz
        =
        \int_{\mathbb R}
        \left(
        \int_U G(t,\bz)\,d\bz
        \right)
        dt .
\]
\end{enumerate}
\end{proposition}

\begin{corollary}\label{cor:vertical-sections}
Let \(U\subseteq\mathbb R^n\) be Lebesgue measurable, and let
\(A\subseteq\mathbb R\times U\) be Lebesgue measurable.  If
\[
        A_{\bz}
        =
        \{t\in\mathbb R:(t,\bz)\in A\},
\]
then
\[
        \mathcal L^{n+1}(A)
        =
        \int_U \mathcal L^1(A_{\bz})\,d\bz .
\]
\end{corollary}

\begin{proof}
Apply Proposition~\ref{prop:tonelli-fubini} to the non-negative measurable
function \(\mathbf 1_A\).  Since
\[
        \mathcal L^{n+1}(A)
        =
        \int_{\mathbb R\times U}\mathbf 1_A(t,\bz)\,dt\,d\bz ,
\]
Tonelli's theorem gives
\[
\begin{aligned}
        \mathcal L^{n+1}(A)
        =
        \int_U
        \left(
        \int_{\mathbb R}\mathbf 1_A(t,\bz)\,dt
        \right)
        d\bz        
        =
        \int_U \mathcal L^1(A_{\bz})\,d\bz .
\end{aligned}
\]
\end{proof}

\begin{proposition}[Fundamental theorem of calculus for absolutely
continuous functions]\cite[Theorem~3.35]{Folland1999}\label{prop:ac-ftc}
Let \(a<b\), and let \(h \) be absolutely continuous.  Then \(h'\) exists almost
everywhere on \([a,b]\), \(h'\in L^1([a,b])\), and
\[
        h(x)-h(a)
        =
        \int_a^x h'(s)\,ds
        \qquad (a\le x\le b).
\]
In particular,
\[
        \int_a^b h'(s)\,ds
        =
        h(b)-h(a).
\]
\end{proposition}

\section{Order properties of the disk model and the RB local approximation}\label{sec:disk-properties}

We first introduce the Lorentz cone and the corresponding cone order.  A subset of an Euclidean space is \textit{closed} if it contains all limits of convergent sequences of its points, and it is \textit{convex} if the line segment between any two of its points is still contained in the subset.  A \textit{cone} is a nonempty closed convex set such that \(\lambda C\subseteq C\) for all \(\lambda\ge0\).  A cone is \emph{pointed} if \(C\cap(-C)=\{0\}\).  

Put
\(
        C=\{(t,\bz)\in\R\times\R^2:t\ge |\bz|\}
\)
and
\(
        H=(-\infty,0]\times\R^2.
\)
This is the standard three-dimensional Lorentz cone, also called the second-order cone; see \cite[Section~2.2.3]{BoydVandenberghe}.  In particular, it is a standard closed convex cone.  Its closedness is also immediate here: it is closed because it is the inverse image of the closed interval \([0,\infty)\) under the continuous function \((t,\bz)\mapsto t-|\bz|\).  It is convex by the triangle inequality: if \((t,\bz),(s,\bw)\in C\) and \(0\le\lambda\le1\), then
\[
        |\lambda\bz+(1-\lambda)\bw|
        \le \lambda|\bz|+(1-\lambda)|\bw|
        \le \lambda t+(1-\lambda)s.
\]
It is pointed, since \((t,\bz),-(t,\bz)\in C\) implies \(t\ge |\bz|\) and \(-t\ge |\bz|\), hence \(t=0\) and \(\bz=0\).  The self-duality of \(C\) is proved in Lemma~\ref{lem:self-dual} below.

The cone order induced by \(C\) is
\[
        \bx\le_C \by
        \quad\Longleftrightarrow\quad
        \by-\bx\in C.
\]
A map \(Q\) defined on a subset of \(\R^3\) is called \(C\)-monotone, if
\[
        \bx\le_C\by
        \quad\Longrightarrow\quad
        Q(\bx)\le_C Q(\by)
\]
whenever both \(\bx\) and \(\by\) lie in its domain.  Equivalently,
\[
        \by-\bx\in C
        \quad\Longrightarrow\quad
        Q(\by)-Q(\bx)\in C.
\]

The representation
\(
        \overline B(\bz,r)\longmapsto (-r,\bz)
\)
identifies the non-bottom part of \(\Disc\) with \(H\), ordered by \(\le_C\). Under this map, the reverse inclusion of disks is equivalent to the Lorentz-cone order. The point \(\be=(1,0,0)\) belongs to \(\Int C\), the interior of \(C\) in the Euclidean Topology. 

Besides, we shall use the following elementary order facts.
\begin{itemize}
\item \textbf{Compact order intervals.}  If \(\bu\le_C \bv, (\bu \neq \bot)\), then
\[
        [\bu,\bv]_C=\{\bh:\bu\le_C \bh\le_C \bv\}
        =(\bu+C)\cap(\bv-C)
\]
is compact.  The equality follows immediately from the definition of \(\le_C\).  The set is closed because \(C\) is closed.  To prove boundedness, write
\(
        \bu=(u_0,\bu'),
        \bv=(v_0,\bv'),
        \bh=(a,\bp).
\)
If \(\bh\in[\bu,\bv]_C\), then
\(
        \bh-\bu\in C
        \text{ and }
        \bv-\bh\in C.
\)
Hence
\(
        a-u_0\ge |\bp-\bu'|
        \text{ and }
        v_0-a\ge |\bv'-\bp|.
\)
In particular,
\(
        u_0\le a\le v_0.
\)
Moreover,
\(
        |\bp-\bu'|\le a-u_0\le v_0-u_0,
\)
and therefore
\(
        |\bp|\le |\bu'|+v_0-u_0.
\)
Thus both the scalar coordinate \(a\) and the spatial coordinate \(\bp\)
are bounded.  Hence \([\bu,\bv]_C\) is bounded.  Since it is also closed in the finite-dimensional Euclidean space \(\R^3\), the Heine--Borel theorem implies that \([\bu,\bv]_C\) is compact.

\item \textbf{Directed suprema and Euclidean convergence.}  Directed subsets of \(H\) have suprema in \(H\), and the corresponding directed nets converge to their suprema in the Euclidean topology.  Indeed, let \(A\subseteq H\) be directed and write its elements as \(\bx=(t_x,\bz_x)\).  The real numbers \(t_x\le0\) are directed upward and bounded above, hence have a supremum \(\tau\le0\).  If \(\bx,\by\in A\) and \(\bg\in A\) is an upper bound of them, then
\[
        |\bz_x-\bz_y|
        \le |\bz_x-\bz_g|+|\bz_g-\bz_y|
        \le (t_g-t_x)+(t_g-t_y).
\]
On a cofinal tail where \(t_x\) and \(t_y\) are close to \(\tau\), the right-hand side is arbitrarily small.  Thus \((\bz_x)\) is a Cauchy net in \(\R^2\), with some limit \(\bz\).  Closedness of the cone order gives \(\bx\le_C(\tau,\bz)\) for all \(\bx\in A\), while every upper bound \(\bw\) of \(A\) satisfies \((\tau,\bz)\le_C \bw\) by passing to the same limit.  Hence \((\tau,\bz)=\sup A\), and the net converges to it.

Finally, for \(\bu,\bv\in H\),
\[
        \bu\ll \bv
        \quad\Longleftrightarrow\quad
        \bv-\bu\in\Int C.
\]
The forward implication follows by applying \(\ll\) to the increasing sequence \(\bv-k^{-1}\be\), whose supremum is \(\bv\).  Conversely, if \(\bv-\bu\in\Int C\) and \(A\) is directed with \(\bv\le_C\sup A=\bs\), then
\[
        \bs-\bu=(\bs-\bv)+(\bv-\bu)\in C+\Int C\subseteq\Int C.
\]
Since the net \(A\) converges to \(\bs\), some element of \(A\) belongs to the open set \(\bu+\Int C\), and therefore lies above \(\bu\).
\end{itemize}

The order on \(\Disc\) becomes particularly transparent after identifying
non-bottom disks with points of the negative Lorentz cone.  We first record
the corresponding compactness and way-below properties, because these are
the only order-theoretic ingredients needed to extract finite-image
approximations from the RB assumption.
In this paper, we will use a contradiction argument, i.e.,  assume that \(\Disc\) is an RB-domain and derive a contradiction.  
%The following lemma is the only point where the RB assumption is used.

\begin{lemma}\label{lem:rb-local}
Assume that \(\Disc\) is an RB-domain.  Let
\(
        P\subseteq (-\infty,0)\times\R^2 = \Int(H)
\)
be a compact rectangle with nonempty interior.  For every \(\varepsilon>0\), there is a finite-image \(C\)-monotone map
\(
        Q_\varepsilon:P\longrightarrow H
\)
such that, for every \(\bx\in P\),
\(
        \bx-\varepsilon\be\le_C Q_\varepsilon(\bx)\le_C \bx.
\)
Consequently,
\(
        \sup_{\bx\in P}\|Q_\varepsilon(\bx)-\bx\|\to0
\)
as \(\varepsilon \rightarrow 0\).
\end{lemma}

\begin{proof}
By the definition of RB-domains, there is a directed family \((r_i)_{i\in I}\) of finite-image Scott-continuous maps on \(\Disc\) such that
\(
        r_i\le \id_\Disc
\)
and
\(
        \bigvee_{i\in I}r_i=\id_\Disc
\)
pointwise.  We identify the non-bottom part of \(\Disc\) with \(H\).

Fix \(\bx\in P\) and define
\[
        \bu_{\bx}=\bx-\frac{3\varepsilon}{4}\be,
        \qquad
        \bv_{\bx}=\bx-\frac{\varepsilon}{2}\be.
\]
Then
\[
        \bv_{\bx}-\bu_{\bx}=\frac{\varepsilon}{4}\be\in\Int C,
\]
so \(\bu_{\bx}\ll \bv_{\bx}\).  Since \(\bv_{\bx}=\bigvee_{i\in I} r_i(\bv_{\bx})\), there is an index \(i_{\bx}\) such that
\(
        \bu_{\bx}\le_C r_{i_{\bx}}(\bv_{\bx}).
\)
Define the Euclidean open neighbourhood
\[
        N_{\bx}=\left\{\by\in\R^3:
        \bx-\by+\frac{\varepsilon}{4}\be\in\Int C,
        \quad
        \by-\bx+\frac{\varepsilon}{2}\be\in\Int C
        \right\}.
\]
It contains \(\bx\).  If \(\by\in P\cap N_{\bx}\), then the two defining inequalities for \(N_{\bx}\) give
\[
        \bu_{\bx}-(\by-\varepsilon\be)
        =\bx-\by+\frac{\varepsilon}{4}\be\in C
\]
and
\[
        \by-\bv_{\bx}=\by-\bx+\frac{\varepsilon}{2}\be\in C.
\]
By the definition of the cone order, this means
\(
        \by-\varepsilon\be\le_C \bu_{\bx}
\)
and
\(
        \bv_{\bx}\le_C \by.
\)
Using monotonicity of \(r_{i_{\bx}}\) and \(r_{i_{\bx}}\le\id_\Disc\), we obtain
\[
        \by-\varepsilon\be
        \le_C \bu_{\bx}
        \le_C r_{i_{\bx}}(\bv_{\bx})
        \le_C r_{i_{\bx}}(\by)
        \le_C \by.
\]

Choose \(\bx_1,\ldots,\bx_m\in P\) such that
\(
        P\subseteq\bigcup_{j=1}^m N_{\bx_j}.
\)
Directedness gives an index \(i\) with
\(
        r_{i_{\bx_j}}\le r_i
\)
for \(j=1,\ldots,m\).  Then for every \(\by\in P\), choosing \(j\) with \(\by\in N_{\bx_j}\), we have
\[
        \by-\varepsilon\be
        \le_C r_{i_{\bx_j}}(\by)
        \le_C r_i(\by)
        \le_C \by.
\]
The lower bound \(\by-\varepsilon\be\) is a non-bottom point of \(H\): if \(\by=(t,\bz)\in P\), then \(t<0\), and hence \(\by-\varepsilon\be=(t-\varepsilon,\bz)\in(-\infty,0)\times\R^2\).  Therefore \(r_i(\by)\) cannot be the bottom element, and so \(r_i(\by)\in H\).  Set
\(
        Q_\varepsilon(\by)=r_i(\by).
\)
This map is finite-image because \(r_i\) has finite image, and it is \(C\)-monotone because \(r_i\) is monotone.

Finally, the inequalities imply
\(
        \by-Q_\varepsilon(\by)\in [0,\varepsilon\be]_C
        =\varepsilon[0,\be]_C.
\)
The interval \([0,\be]_C\) is compact, hence bounded.  Thus there is a constant \(M>0\), independent of \(\by\) and \(\varepsilon\), such that
\[
        \|\by-Q_\varepsilon(\by)\|\le M\varepsilon
        \qquad (\by\in P).
\]
Therefore
\(
        \sup_{\by\in P}\|Q_\varepsilon(\by)-\by\|\le M\varepsilon\to0
\)
as \(\varepsilon\rightarrow 0\).
\end{proof}

\section{Lorentz cone and matrix preliminaries}\label{sec:cone-functional}

In this section, we record some linear-algebra and convex-geometric facts about the Lorentz cone and the corresponding matrix cone.

The scalar product on \(\R^3\) is the usual one.  For vectors \(\ba,\bb\in\R^3\), the symbol
\(
        \ba\otimes \bb
\)
means the rank-one matrix
\[
        \ba\otimes \bb=\ba\bb^T,
        \qquad
        (\ba\otimes \bb)_{ij}=a_i b_j.
\]
Equivalently, \((\ba\otimes \bb)\bh=(\bb\cdot \bh)\ba\).  We write
\(
M_3(\mathbb R)
=
\left\{(a_{ij})_{0\le i,j\le 2}:a_{ij}\in\mathbb R\right\}
\)
for the real vector space of all \(3\times3\) real matrices.  Thus \(M_3(\R)\cong\End(\R^3)\), and it has dimension \(9\).

The dual cone of \(C\) is
\(
        C^*=\{\bw\in\R^3:\bw\cdot \bv\ge0\text{ for all }\bv\in C\}.
\) The Lorentz cone \(C\) is self-dual, i.e., \(C^*=C\).  This is a standard fact.  For completeness, we give a short proof.
\begin{lemma}\label{lem:self-dual}
The Lorentz cone \(C\) is self-dual: \(C^*=C\).
\end{lemma}

\begin{proof}
If \(\bv=(v_0,\bv')\in C\) and \(\bw=(w_0,\bw')\in C\), then
\[
        \bv\cdot \bw=v_0w_0+\bv'\cdot \bw'
        \ge v_0w_0-|\bv'||\bw'|
        \ge0.
\]
Hence \(C\subseteq C^*\).  Conversely, if \(\bw\notin C\), then \(w_0<|\bw'|\).  For \(\bw'\ne0\), take \(\bv=(|\bw'|,-\bw')\in C\).  Then
\[
        \bv\cdot \bw=|\bw'|w_0-|\bw'|^2<0.
\]
If \(\bw'=0\) and \(w_0<0\), take \(\bv=(1,0,0)\).  Thus \(\bw\notin C^*\).  Therefore \(C^*=C\).
\end{proof}

The cone of a set \(A\subseteq M_3(\R)\) is
\(
        \cone(A)=\left\{\sum_{j=1}^m \lambda_j A_j:m\in\N,\ \lambda_j\ge0,\ A_j\in A\right\}.
\)      
Since $C$ is a closed convex cone, we can define the matrix cone generated by $C$
as follows:
\[
        \RC=\cone\{\bv\otimes \bw:\bv,\bw\in C\}\subseteq M_3(\R).
\]
Here 
 \(T\in\RC\) if and only if
\[
        T=\sum_{j=1}^m \lambda_j \bv_j\otimes \bw_j
\]
for some \(m\), \(\lambda_j\ge0\), and \(\bv_j,\bw_j\in C\).  Since \(\lambda_j\bv_j\in C\), the coefficients may be absorbed into the first vector.

A convex combination of points \(\ba_0,\ldots,\ba_m\) in a vector space is a linear combination
\(
        \sum_{j=0}^m \lambda_j \ba_j            
\)
with \(\lambda_j\ge0\) and \(\sum_{j=0}^m \lambda_j=1\).  The convex hull of a set \(A\) is the set of all convex combinations of points in \(A\):
\[
        \conv(A)=\left\{\sum_{j=0}^m \lambda_j \ba_j:m\in\N,\ \lambda_j\ge0,\ \sum_{    j=0}^m \lambda_j=1,\ \ba_j\in A\right\}.
\]      
We use Caratheodory's theorem in the following standard form.
\begin{theorem}[Caratheodory's theorem] \cite[Theorem~17.1]{Rockafellar1970}\label{thm:caratheodory}
Let \(A\subseteq \R^d\).  If \(\bx\in\conv(A)\), then there exist points \(\ba_0,\ldots,\ba_d\in A\) and coefficients \(\lambda_0,\ldots,\lambda_d\ge0\) such that
\[
        \sum_{i=0}^d \lambda_i=1,
        \qquad
        \bx=\sum_{i=0}^d \lambda_i \ba_i.
\]
Equivalently, every point of \(\conv(A)\) is a convex combination of at most \(d+1\) points of \(A\).
\end{theorem}

\begin{lemma}\label{lem:compact-base-cone}
Let \(E\) be a finite-dimensional real normed space.  Let \(S\subseteq E\) be compact, and suppose that there is a continuous linear functional \(\Phi:E\to\R\) such that
\(
        \Phi(s)=1
\)
for all \(s\in S\).  Then
\(
        \{rs:r\ge0,\ s\in S\}
\)
is closed in \(E\).
\end{lemma}

\begin{proof}
Let \(r_n\bs_n\to \bx\), where \(r_n\ge0\) and \(\bs_n\in S\).  Applying \(\Phi\), we get
\(
        r_n=\Phi(r_n\bs_n)\to \Phi(\bx).
\)
If \(\Phi(\bx)=0\), then \(r_n\to0\).  Since \(S\) is compact, it is bounded, and hence \(r_n\bs_n\to0\).  Thus \(\bx=0\), which belongs to $\{rs:r\ge0,\ s\in S\}$.  If \(\Phi(\bx)>0\), then \(r_n>0\) for all large \(n\), and
\[
        \bs_n=\frac{r_n\bs_n}{r_n}\to \frac{\bx}{\Phi(\bx)}.
\]
Since \(S\) is closed, \(\bx/\Phi(\bx)\in S\), and so \(\bx\) belongs to $\{rs:r\ge0,\ s\in S\}$.
\end{proof}

In the following, we show  that the cone \(\RC\) is closed. 

\begin{proposition}\label{lem:RC-closed}
The cone \(\RC\) is closed in \(M_3(\R)\).
\end{proposition}

\begin{proof}
Let
\(
        K=\{\bv\in C:v_0=1\}=\{(1,\bz):|\bz|\le1\}.
\)
This is a compact disk.  The map
\(
        K\times K\to M_3(\R): (\bp,\bq)\mapsto \bp\otimes \bq
\)
is continuous, since in coordinates \(\bp\otimes \bq=(p_iq_j)_{0\le i,j\le2}\).  Hence
\(
        A=\{\bp\otimes \bq:\bp,\bq\in K\}
\)
is compact.

By Caratheodory's theorem in the nine-dimensional space \(M_3(\R)\), every element of \(\conv A\) is a convex combination of at most \(10\) elements of \(A\).  Hence \(\conv A\) is the image of the compact set
\[
        \Delta_9\times A^{10},
        \qquad
        \Delta_9=\{(\lambda_0,\ldots,\lambda_9):\lambda_i\ge0,
        \ \sum_{i=0}^9\lambda_i=1\},
\]
under the continuous map
\[
        (\lambda_0,\ldots,\lambda_9,A_0,\ldots,A_9)
        \longmapsto
        \sum_{i=0}^9\lambda_iA_i.
\]
Therefore
\(
        S=\conv A
\)
is compact.

Every nonzero \(\bv\in C\) can be written uniquely as
\(\bv=v_0\bp\), where \(v_0>0\) and \(\bp\in K\).  Hence each generator
\(\bv\otimes\bw\) of \(\RC\) is a non-negative scalar multiple of some
\(\bp\otimes\bq\in A\).  Therefore every finite conic combination of
generators has the form
\[
        \sum_{j=1}^n a_j\,\bv_j\otimes\bw_j
        =
        r\sum_{j=1}^n\lambda_j\,\bp_j\otimes\bq_j,
\]
where \(r\ge0\), \(\lambda_j\ge0\), \(\sum_j\lambda_j=1\), and
\(\bp_j\otimes\bq_j\in A\).  Thus it belongs to
\(\{rB:r\ge0,\ B\in S\}\).  Conversely, every element \(rB\), with
\(B\in S=\operatorname{conv}(A)\), is by definition a finite conic
combination of generators.  Hence
\[
        \RC=\{rB:r\ge0,\ B\in S\}.
\]

% Every finite conic combination of generators can be normalized into a scalar multiple of a convex combination of elements of \(A\), and conversely any \(rB\) with \(B=\sum\lambda_j\bp_j\otimes \bq_j\in S\) belongs to \(\RC\). Therefore,
% every nonzero \(\bv\in C\) can be written uniquely as \(\bv=v_0\bp\), with \(v_0>0\) and \(\bp\in K\), and then
% \[
%         \RC=\{rB:r\ge0,\ B\in S\}.
% \]

Given $T \in M_3 (\mathbb{R})$, let \(\Phi(T)=\be\cdot T\be\), where \(\be=(1,0,0)\).  For every \(\bp,\bq\in K\), we have 
\(
        \Phi(\bp\otimes \bq)=(\be\cdot \bp)(\bq\cdot \be)=1.
\)
Hence \(\Phi(B)=1\) for all \(B\in S\).  Lemma \ref{lem:compact-base-cone} now gives that \(\RC\) is closed.
\end{proof}

%We also need a simple stability property of closed convex cones under integration.  It says that if a matrix-valued integrand lies in \(\RC\)
%almost everywhere, then its integral also lies in \(\RC\).  This will be applied to the blockwise representation of finite-image monotone maps.
The following lemma is a  consequence of the finite-dimensional strong separation theorem for closed convex cones; see \cite[Corollary~11.4.1]{Rockafellar1970}.

\begin{lemma}\label{lem:cone-integral}
Let \(U\subseteq\mathbb R^n\) be a Lebesgue measurable set, and let
\(K\subseteq M_3(\mathbb R)\) be a closed convex cone.  If
\(F:U\to M_3(\mathbb R)\) is integrable and
\[
        F(\bz)\in K
\]
for almost every \(\bz\in U\), then
\[
        \int_U F(\bz)\,d\bz\in K.
\]
\end{lemma}

\begin{proof}
Put
\[
        A=\int_U F(\bz)\,d\bz .
\]
Suppose, for a contradiction, that \(A\notin K\).  Since \(M_3(\mathbb R)\)
is finite-dimensional and \(K\) is a closed convex set, the strong separation theorem gives a nonzero linear functional \(L:M_3(\mathbb R)\to\mathbb R\) and a real number \(\alpha\) such that
\[
        L(A)<\alpha\le L(B)\qquad\text{for all }B\in K .
\]
Because \(K\) is a cone, \(0\in K\), and hence \(\alpha\le L(0)=0\).
Moreover, for every \(B\in K\) and every \(t>0\), we have \(tB\in K\), so
\(
        \alpha\le L(tB)=tL(B).
\)
If \(L(B)<0\) for some \(B\in K\), then \(tL(B)\to-\infty\) as
\(t\to\infty\), contradicting \(\alpha\le tL(B)\).  Therefore
\[
        L(B)\ge0\qquad\text{for all }B\in K .
\]
Since \(\alpha\le0\), the separation inequality also gives
\[
        L(A)<\alpha\le0.
\]
Thus we have found a linear functional \(L\) such that \(L\ge0\) on \(K\)
but \(L(A)<0\).

Since \(F(\bz)\in K\) almost everywhere, it follows that
\[
        L(F(\bz))\ge0
\]
almost everywhere.  Matrix-valued integrals here are ordinary coordinatewise
Lebesgue integrals.  Writing \(F(\bz)=(F_{ij}(\bz))\) and
\[
        L(B)=\sum_{i,j}c_{ij}B_{ij},
\]
linearity gives
\[
\begin{aligned}
        L\left(\int_U F(\bz)\,d\bz\right)
        =
        \sum_{i,j}c_{ij}\int_U F_{ij}(\bz)\,d\bz  
        =
        \int_U\sum_{i,j}c_{ij}F_{ij}(\bz)\,d\bz  
        =
        \int_U L(F(\bz))\,d\bz
        \ge0.
\end{aligned}
\]
This contradicts \(L(A)<0\).  Hence \(A\in K\).
\end{proof}
% \begin{lemma}\label{lem:cone-integral}

% Let \(U\subseteq\mathbb R^n\) be a Lebesgue measurable set, and let
% \(K\subseteq M_3(\mathbb R)\) be a closed convex cone.  If
% \(F:U\to M_3(\mathbb R)\) is integrable and
% \(
%         F(\bz)\in K
% \)
% for almost every \(\bz\in U\), then
% \[
%         \int_U F(\bz)\,d\bz\in K.
% \]
% \end{lemma}

% \begin{proof}
% Suppose the integral were not in \(K\).  By the finite-dimensional separation theorem, there would be a linear functional \(L\) on \(M_3(\R)\) such that \(L\ge0\) on \(K\), but
% \[
%         L\left(\int_U F(\bz)\dd \bz\right)<0.
% \]
% Since \(F(\bz)\in K\) almost everywhere, \(L(F(\bz))\ge0\) almost everywhere.  Matrix-valued integrals here are ordinary coordinatewise Lebesgue integrals.  Writing \(F(\bz)=(F_{ij}(\bz))\) and \(L(A)=\sum_{i,j}c_{ij}A_{ij}\), linearity gives
% \[
%         L\left(\int_U F(\bz)\dd \bz\right)
%         =\sum_{i,j}c_{ij}\int_U F_{ij}(\bz)\dd \bz
%         =\int_U\sum_{i,j}c_{ij}F_{ij}(\bz)\dd \bz
%         =\int_U L(F(\bz))\dd \bz\ge0,
% \]
% a contradiction.
% \end{proof}

We shall also use compactly supported test functions.  If \(\Omega\subseteq\R^3\) is open, then \(C_c^\infty(\Omega)\) denotes the space of smooth real-valued functions \(\psi:\Omega\to\R\) whose support is compact and contained in \(\Omega\).  We write a point of \(\R^3\) as \(\bx=(t,\bz)\), with \(\bz=(z_1,z_2)\in\R^2\), and we write \(\dd \bx=\dd t\,\dd z_1\,\dd z_2\) and \(\dd \bz=\dd z_1\,\dd z_2\).  Also,
\[
        \nabla\psi=(\partial_t\psi,\partial_{z_1}\psi,\partial_{z_2}\psi),
        \qquad
        \nabla_\bz\psi=(\partial_{z_1}\psi,\partial_{z_2}\psi).
\]
For a finite-image map \(Q:\Omega\to\R^3\), define the ordinary matrix integral if it is integrable:
\[
        T_Q(\psi)=-\int_\Omega Q(\bx)\otimes\nabla\psi(\bx)\dd \bx,
\]
which means entrywise Lebesgue integration:
\[
        (T_Q(\psi))_{ij}
        =-\int_\Omega Q_i(\bx)\,\partial_j\psi(\bx)\dd \bx.
\]
Thus \(\bx\mapsto Q(\bx)\otimes\nabla\psi(\bx)\) is simply a \(3\times3\) matrix-valued function.

We will use Rademacher's theorem in the standard Euclidean form.
\begin{theorem}[Rademacher's theorem]\cite[Theorem 3.2]{EvansGariepy2015}\label{thm:rademacher}
Let \(U\subseteq\R^n\) be open and let \(f:U\to\R^m\) be Lipschitz.  Then \(f\) is differentiable almost everywhere in \(U\) with respect to Lebesgue measure.  In particular, if \(f:U\subseteq\R^2\to\R\) is \(1\)-Lipschitz, then \(\nabla f(\bz)\) exists for almost every \(\bz\) and \(|\nabla f(\bz)|\le1\) there.
\end{theorem}

\section{The planar closed disk model is not an RB-domain}\label{sec:proof}

In this section we prove that Lawson's planar closed-disk example is not
an RB-domain.  The proof has two parts.  First, we show that every
finite-image \(C\)-monotone map gives, after testing against a nonnegative
smooth function, a matrix in \(\RC\).  Second, applying this to the
finite-image approximants coming from a hypothetical RB structure forces
the identity matrix \(I_3\) to belong to \(\RC\), which is impossible.

Throughout this section let
\(
        \Omega=(\alpha,\beta)\times U\subseteq\R\times\R^2
\)
where \(U\subseteq\R^2\) is open.  A set \(E\subseteq\Omega\) is called \(C\)-upper in \(\Omega\) if it is an upper set under the cone order \(\le_C\), i.e.,
\(
        \bx\in E,
         \by\in\Omega,
         \by-\bx\in C
        \quad\Longrightarrow\quad
        \by\in E.
\)
For sets \(E,F\), their symmetric difference is
\(
        E\triangle F=(E\setminus F)\cup(F\setminus E).
\)
The graph of a function \(f:U\to\R\) is
\(
        \graph(f)=\{(f(\bz),\bz):\bz\in U\}\subseteq\R\times\R^2.
\)

\begin{lemma}\label{lem:upper-epigraph}
Let \(E\subseteq\Omega\) be \(C\)-upper.  For \(\bz\in U\), define
\(
        f_E(\bz)=\inf\left(\{t\in(\alpha,\beta):(t,\bz)\in E\}\cup\{\beta\}\right).
\)
Then \(f_E:U\to[\alpha,\beta]\) is \(1\)-Lipschitz.  Moreover,
\(
        E\triangle\{(t,\bz)\in\Omega:t>f_E(\bz)\}
\)
is contained in \(\graph(f_E)\), and hence has three-dimensional Lebesgue measure zero.  In particular, \(E\) is Lebesgue measurable.
\end{lemma}

\begin{proof}
For fixed \(\bz\), the vertical section
\[
        E_{\bz}=\{t\in(\alpha,\beta):(t,\bz)\in E\}
\]
is an upper interval, because \((s,\bz) - (t,\bz) = (s-t,0)\in C\) whenever \(s\ge t\).  Thus, if \(t>f_E(\bz)\), then \((t,\bz)\in E\); and if \(t<f_E(\bz)\), then \((t,\bz)\notin E\).  Therefore \(E\) and the set \(G_{f_E}=\{(t,\bz)\in\Omega:t>f_E(\bz)\}\) can differ only when \(t=f_E(\bz)\), namely on \(\graph(f_E)\cap\Omega\).

We prove the Lipschitz estimate.  Let \(\bz,\bz'\in U\) and set \(r=|\bz-\bz'|\).  If \(f_E(\bz)+r\ge\beta\), then
\[
        f_E(\bz')\le\beta\le f_E(\bz)+r.
\]
Otherwise choose \(t\) with
\[
        f_E(\bz)<t<\beta-r.
\]
Then \((t,\bz)\in E\), and
\[
        (t+r,\bz')-(t,\bz)=(r,\bz'-\bz)\in C.
\]
Hence \((t+r,\bz')\in E\), and so \(f_E(\bz')\le t+r\).  Letting \(t\downarrow f_E(\bz)\), we get
\[
        f_E(\bz')\le f_E(\bz)+|\bz-\bz'|.
\]
Interchanging \(\bz\) and \(\bz'\) gives \(|f_E(\bz)-f_E(\bz')|\le |\bz-\bz'|\).  Hence \(f_E\) is continuous.

Let \(\mathcal L^k\) denote \(k\)-dimensional Lebesgue measure.  Since
\(f_E\) is continuous, \(\operatorname{graph}(f_E)\) is a Borel subset of
\(\mathbb R\times U\).  For each fixed \(\bz\in U\), the vertical section
of the graph is the singleton
\[
        \{t\in\mathbb R:(t,\bz)\in\operatorname{graph}(f_E)\}
        =
        \{f_E(\bz)\}.
\]
Hence this section has one-dimensional Lebesgue measure zero:
\[
        \mathcal L^1(\{f_E(\bz)\})=0.
\]
% By Tonelli-Fubini's theorem~\cite[Theorem~2.37]{Folland1999},
% \[
%         \mathcal L^3(\operatorname{graph}(f_E))
%         =
%         \int_U
%         \mathcal L^1\bigl(\{t\in\mathbb R:(t,\bz)\in
%         \operatorname{graph}(f_E)\}\bigr)\,d\bz
%         =
%         \int_U \mathcal L^1(\{f_E(\bz)\})\,d\bz
%         =
%         0.
% \]
Let
\(
        \Gamma=\operatorname{graph}(f_E).
\)
Since \(f_E\) is continuous, \(\Gamma\) is Borel, hence
\(\mathbf 1_\Gamma\) is measurable.  By the definition of Lebesgue measure,
\[
        \mathcal L^3(\Gamma)
        =
        \int_{\mathbb R\times U}\mathbf 1_\Gamma(t,\bz)\,dt\,d\bz .
\]
Applying Tonelli-Fubini's theorem  (see Proposition \ref{prop:tonelli-fubini})
%\cite[Theorem~2.37]{Folland1999} 
to the nonnegative measurable function
\(\mathbf 1_\Gamma\), we obtain
\[
\begin{aligned}
        \mathcal L^3(\Gamma)
        =
        \int_U
        \left(
        \int_{\mathbb R}\mathbf 1_\Gamma(t,\bz)\,dt
        \right)
        d\bz  
=
        \int_U
        \mathcal L^1
        \bigl(\{t\in\mathbb R:(t,\bz)\in\Gamma\}\bigr)
        d\bz .
\end{aligned}
\]
But for each \(\bz\in U\),
\[
        \{t\in\mathbb R:(t,\bz)\in\Gamma\}
        =
        \{f_E(\bz)\},
\]
and every singleton in \(\mathbb R\) has one-dimensional Lebesgue measure
zero.  Therefore
\[
        \mathcal L^3(\Gamma)
        =
        \int_U \mathcal L^1(\{f_E(\bz)\})\,d\bz
        =
        0.
\]
Set
\(
        G_{f_E}
        =
        \{(t,\bz)\in\Omega:t>f_E(\bz)\}.
\)
Since \(f_E\) is continuous, the function
\(
        (t,\bz)\longmapsto t-f_E(\bz)
\)
is continuous.  Hence
\(
        G_{f_E}
        =
        \{(t,\bz):t-f_E(\bz)>0\}
\)
is open in \(\Omega\), and therefore is a Borel subset of \(\mathbb R\times U\).
Since \(E\triangle G_{f_E}\subseteq \operatorname{graph}(f_E)\) and \(\operatorname{graph}(f_E)\) is a null set, the symmetric
difference \(E\triangle G_{f_E}\) is null.  Thus \(E\) differs from the
Borel set \(G_{f_E}\) by a null set.  Since Lebesgue measure is complete, \(E\) is Lebesgue
measurable.
\end{proof}

% The following identity is the elementary epigraph case of the Gauss--Green formula for Lipschitz graphs. 
%  We give the proof in order to keep the argument self-contained; it uses only Fubini's theorem, the one-dimensional fundamental theorem of calculus, and Rademacher's theorem.

The preceding lemma reduces \(C\)-upper sets to Lipschitz epigraphs.  The next formula computes the integral of the gradient of a test function over such an epigraph and rewrites it as an integral over the graph \(t=f(\bz)\).  This is the step that produces the cone vector
\((1,-\nabla f)\) we need in the following, which is the elementary epigraph case of the Gauss--Green formula for Lipschitz graphs. 

\begin{lemma}\label{lem:epigraph-integration}
Let \(U\subseteq\R^2\) be open, let \(\Omega=(\alpha,\beta)\times U\), and let \(f:U\to[\alpha,\beta]\) be \(1\)-Lipschitz.  Put
\[
        G_f=\{(t,\bz)\in\Omega:t>f(\bz)\}.
\]
For every \(\psi\in C_c^\infty(\Omega)\), one has
\[
        -\int_{G_f}\nabla\psi(t,\bz)\dd t\dd \bz
        =
        \int_U \psi(f(\bz),\bz)(1,-\nabla f(\bz))\dd \bz.
\]
Here \(\psi\) is understood as extended by zero outside \(\Omega\), and \(\nabla f(\bz)\) is defined for almost every \(\bz\in U\) by Rademacher's theorem.
\end{lemma}

\begin{proof}
Extend \(\psi\) by zero outside \(\Omega\).  The formula is a vector identity in \(\R^3=\R\times\R^2\), so we prove its time component and spatial components separately.

First consider the time component.  For each fixed \(\bz\in U\), the function \(t\mapsto\psi(t,\bz)\) is smooth and vanishes near \(t=\beta\).  Hence the fundamental theorem of calculus gives
\[
        -\int_{f(\bz)}^\beta \partial_t\psi(t,\bz)\dd t
        =-\bigl(\psi(\beta,\bz)-\psi(f(\bz),\bz)\bigr)
        =\psi(f(\bz),\bz).
\]
Integrating in \(\bz\), we obtain
\[
        -\int_U\int_{f(\bz)}^\beta \partial_t\psi(t,\bz)\dd t\dd \bz
        =\int_U\psi(f(\bz),\bz)\dd \bz.
\]
This is the first component of the desired formula.

We now treat the two spatial components.  Define
\[
        H(\bz)=\int_{f(\bz)}^\beta \psi(t,\bz)\dd t.
\]
We first verify that \(H\) is Lipschitz and compactly supported in \(U\).  Put
\(
        M_0=\|\psi\|_\infty
\)
and
\(
        M_1=\|\nabla_\bz\psi\|_\infty.
\)
For \(\bz,\bz'\in U\), using oriented integrals, we write
\[
\begin{aligned}
H(\bz)-H(\bz')
&=\int_{f(\bz)}^\beta\psi(t,\bz)\dd t
  -\int_{f(\bz')}^\beta\psi(t,\bz')\dd t  \\
&=\int_{f(\bz)}^\beta\bigl(\psi(t,\bz)-\psi(t,\bz')\bigr)\dd t
  +\int_{f(\bz)}^{f(\bz')}\psi(t,\bz')\dd t.
\end{aligned}
\]
Therefore
\[
        |H(\bz)-H(\bz')|
        \le (\beta-\alpha)M_1|\bz-\bz'|+M_0|f(\bz)-f(\bz')|
        \le ((\beta-\alpha)M_1+M_0)|\bz-\bz'|.
\]
Thus \(H\) is Lipschitz.  Let \(\pi_U:\Omega\to U\) be the projection \(\pi_U(t,\bz)=\bz\).  If \(\bz\notin\pi_U(\supp\psi)\), then \(\psi(t,\bz)=0\) for every \(t\), and hence \(H(\bz)=0\).  Thus
\[
        \{\bz:H(\bz)\ne0\}\subseteq \pi_U(\supp\psi).
\]
Since \(\supp\psi\Subset\Omega\) is compact and \(\pi_U\) is continuous, \(\pi_U(\supp\psi)\Subset U\) is compact.  Hence \(H\) has compact support in \(U\).

By Theorem \ref{thm:rademacher}, \(f\) is differentiable for almost every \(\bz\in U\).  At such a point, the Lipschitz rule for differentiating an integral with a variable lower limit gives
\[
        \nabla H(\bz)
        =\int_{f(\bz)}^\beta \nabla_\bz\psi(t,\bz)\dd t
        -\psi(f(\bz),\bz)\nabla f(\bz). \tag{*}
\]
The second term is the contribution from differentiating the lower limit \(f(\bz)\).  Since \(H\) is Lipschitz and compactly supported in \(U\), its almost-everywhere gradient is integrable.

Extend \(H\) by zero to \(\mathbb R^2\), and still denote the extension by
\(H\).  Since \(\operatorname{supp}H\Subset U\), this extension is compactly
supported and Lipschitz on \(\mathbb R^2\).  
For almost every fixed
\(z_2\), the function \(h :z_1\mapsto H(z_1,z_2)\) is compactly supported and
Lipschitz, with derivative \(\partial_{z_1}H(z_1,z_2)\) almost everywhere.
Choose \(a<b\) such that \(h(z_1)=0\) for \(z_1\notin[a,b]\).  
By the one-dimensional fundamental theorem of calculus for absolutely continuous
functions,
\[
        \int_{\mathbb R} h'(z_1)\,d z_1
        =
        \int_a^b h'(z_1)\,d z_1
        =
        h(b)-h(a)
        =
        0
\]
for almost every \(z_2\).  
% Since \(\partial_{z_1}H\) is integrable, Fubini's
% theorem yields
% \[
%         \int_{\mathbb R^2}\partial_{z_1}H(\bz)\,d\bz=0.
% \]

Since \(H\) is compactly supported on \(\mathbb R^2\), choose a rectangle
\([a,b]\times[c,d]\) such that
\[
        \operatorname{supp}H\subseteq [a,b]\times[c,d].
\]
Since \(\partial_{z_1}H\in L^1(\mathbb R^2)\), Fubini's theorem gives
\[
\begin{aligned}
        \int_{\mathbb R^2}\partial_{z_1}H(\bz)\,d\bz
        &=
        \int_{\mathbb R}
        \left(
        \int_{\mathbb R}
        \partial_{z_1}H(z_1,z_2)\,d z_1
        \right)
        d z_2 .
\end{aligned}
\]
For almost every \(z_2\), the slice
\[
        h_{z_2}(z_1)=H(z_1,z_2)
\]
is a Lipschitz function on \(\mathbb R\), is supported in \([a,b]\), and
satisfies
\[
        h_{z_2}'(z_1)=\partial_{z_1}H(z_1,z_2)
        \quad\text{for almost every }z_1.
\]
Hence \(h_{z_2}\) is absolutely continuous on \([a,b]\), and the
fundamental theorem of calculus for absolutely continuous functions gives
\[
\begin{aligned}
        \int_{\mathbb R}
        \partial_{z_1}H(z_1,z_2)\,d z_1=
        \int_a^b h_{z_2}'(z_1)\,d z_1  
        =h_{z_2}(b)-h_{z_2}(a)
        =0 .
\end{aligned}
\]
Therefore
\[
        \int_{\mathbb R^2}\partial_{z_1}H(\bz)\,d\bz
        =
        \int_{\mathbb R}0\,d z_2
        =
        0.
\]
The same argument in the \(z_2\)-direction gives
\[
        \int_{\mathbb R^2}\partial_{z_2}H(\bz)\,d\bz=0.
\]
Since \(H\) is supported in \(U\), the same identities hold over \(U\).
Therefore,
\[
        \int_U \nabla H(\bz)\,d\bz=0.
\]
Integrating the above formula 
and formula $(*)$ we get
\[
        0
        =\int_U\int_{f(\bz)}^\beta \nabla_\bz\psi(t,\bz)\dd t\dd \bz
        -\int_U\psi(f(\bz),\bz)\nabla f(\bz)\dd \bz.
\]
Therefore
\[
        -\int_U\int_{f(\bz)}^\beta \nabla_\bz\psi(t,\bz)\dd t\dd \bz
        =-
        \int_U\psi(f(\bz),\bz)\nabla f(\bz)\dd \bz.
\]
Combining this spatial identity with the time component gives the formula.
\end{proof}

We now apply the preceding epigraph formula to finite-image
\(C\)-monotone maps.  The level sets of such a map are \(C\)-upper, and
therefore each level contributes a rank-one matrix generated by two
vectors in \(C\).  The following proposition shows that the total tested
matrix necessarily belongs to \(\RC\).

\begin{proposition}\label{prop:finite-image-RC}
Let \(Q:\Omega\to\R^3\) be finite-image and \(C\)-monotone.  Then \(Q\) is Lebesgue measurable.  Moreover, for every non-negative \(\psi\in C_c^\infty(\Omega)\), the matrix integral \(T_Q(\psi)\) is well-defined and
\[
        T_Q(\psi)\in\RC.
\]
\end{proposition}

\begin{proof}
Let the distinct values of \(Q\) be \(\bq_1,\ldots,\bq_N\).  Choose \(\blambda\in\Int C\) outside the finitely many hyperplanes
\[
        \{\bmu\in\R^3:\bmu\cdot(\bq_i-\bq_j)=0\},
        \qquad i<j.
\]
Then the real numbers \(\blambda\cdot \bq_j\) are pairwise distinct.  Since \(C\) is self-dual, \(\blambda\in C^*\).  Reorder the values so that
\[
        \blambda\cdot \bq_1<\blambda\cdot \bq_2<\cdots<\blambda\cdot \bq_N.
\]
Choose numbers \(s_i\) with
\[
        \blambda\cdot \bq_i<s_i<\blambda\cdot \bq_{i+1}
        \qquad(1\le i<N),
\]
and define
\[
        E_i=\{\bx\in\Omega:\blambda\cdot Q(\bx)>s_i\}.
\]
Then \[
% E_N \subseteq 
E_{N-1} \subseteq \dots \subseteq E_2\subseteq E_1.\]
If \(\bx\in E_i\), \(\by\in\Omega\), and \(\bx\leq_C \by\), then \(Q(\by)-Q(\bx)\in C\), because \(Q\) is \(C\)-monotone.  Since \(\blambda\in C^*\),
\[
        \blambda\cdot Q(\by)\ge \blambda\cdot Q(\bx)>s_i.
\]
Thus \(\by\in E_i\).  Hence every \(E_i\) is \(C\)-upper, and Lemma \ref{lem:upper-epigraph} implies that \(E_i\) is Lebesgue measurable.

The layer decomposition
\[
        Q=\bq_1+\sum_{i=1}^{N-1}(\bq_{i+1}-\bq_i)\one_{E_i}
\]
holds pointwise: if \(Q(\bx)=\bq_j\), then \(\bx\in E_i\) exactly for \(i<j\).  This decomposition shows that \(Q\) is Lebesgue measurable.  Indeed, each \(E_i\) is Lebesgue measurable, and hence each indicator function \(\one_{E_i}\) is measurable.  Since \(Q\) is a finite linear combination of such indicator functions with vector coefficients, \(Q\) is measurable.  Moreover, \(Q\) is finite-image, so
\[
        \|Q(\bx)\|\le \max_{1\le i\le N}\|\bq_i\|<\infty.
\]
Since \(\nabla\psi\) is smooth and compactly supported, the matrix-valued function \(Q\otimes\nabla\psi\) is integrable.  Therefore \(T_Q(\psi)\) is an ordinary matrix integral.

Let \(f_i=f_{E_i} =\inf(\{t \in (\alpha,\beta): (t,\bz) \in E\} \cup \{\beta\})\) as defined in Lemma~\ref{lem:upper-epigraph}, and put \(G_{f_i}=\{(t,\bz)\in\Omega:t>f_i(\bz)\}\).  Since \(E_{i+1}\subseteq E_i\), we have
\[
        f_1\le f_2\le\cdots\le f_{N-1}.
\]
Using the layer decomposition, we obtain
\[
\begin{aligned}
        T_Q(\psi)
        &=-\int_\Omega \bq_1\otimes\nabla\psi\dd \bx
          -\sum_{i=1}^{N-1}(\bq_{i+1}-\bq_i)\otimes\int_{E_i}\nabla\psi\dd \bx.
\end{aligned}
\]
The first term is zero because \(\bq_1\) is constant and \(\int_\Omega\nabla\psi\dd \bx=0\).  For each \(i\), Lemma \ref{lem:upper-epigraph} gives that \(G_{f_i} \subseteq E_i \), and the difference is null; hence
\[
        \int_{E_i}\nabla\psi\dd \bx
        =\int_{G_{f_i}}\nabla\psi\dd \bx.
\]
Applying Lemma \ref{lem:epigraph-integration} gives
\[
        -\int_{E_i}\nabla\psi\dd \bx
        =\int_U\psi(f_i(\bz),\bz)(1,-\nabla f_i(\bz))\dd \bz.
\]
Therefore
\[
\begin{aligned}
        T_Q(\psi)
        &=\sum_{i=1}^{N-1}
        (\bq_{i+1}-\bq_i)\otimes
        \int_U\psi(f_i(\bz),\bz)(1,-\nabla f_i(\bz))\dd \bz \\
        &=\sum_{i=1}^{N-1}\int_U
        \psi(f_i(\bz),\bz)(\bq_{i+1}-\bq_i)\otimes(1,-\nabla f_i(\bz))\dd \bz.
\end{aligned}
\]
The last equality is just the coordinatewise linearity of the matrix integral.

Fix a point \(\bz\in U\) at which all the \(f_i\) are differentiable.  This holds by Theorem~\ref{thm:rademacher} for almost every \(\bz\), because there are only finitely many Lipschitz functions.  At the fixed point \(\bz\), the numbers
\[
        f_1(\bz),\ldots,f_{N-1}(\bz)
\]
are ordered increasingly.  We group together each maximal consecutive
string of equal values.  Thus a block \(r,\ldots,s\) means that
\[
        f_r(\bz)=f_{r+1}(\bz)=\cdots=f_s(\bz)=\tau,
\]
and that the block cannot be enlarged; equivalently, if \(r>1\) then
\(f_{r-1}(\bz)<\tau\), and if \(s<N-1\) then
\(f_{s+1}(\bz)>\tau\).
Since the functions are ordered, differentiability at a common contact point implies
\[
        \nabla f_r(\bz)=\cdots=\nabla f_s(\bz).
\]
Indeed, if \(g\le h\), \(g(\bz)=h(\bz)\), and both functions are differentiable at \(\bz\), then \(h-g\) has a local minimum at \(\bz\), so \(\nabla h(\bz)=\nabla g(\bz)\).  Applying this to consecutive pairs gives the displayed equality.

The contribution of the block \(r,\ldots,s\) is
\[
        \psi(\tau,\bz)(\bq_{s+1}-\bq_r)\otimes(1,-\nabla f_r(\bz)). \tag{\(\ast\)}
\]
If \(\tau=\alpha\) or \(\tau=\beta\), then \(\psi(\tau,\bz)=0\) by the zero-extension convention and compact support.  
Otherwise choose real numbers \(t_-\) and \(t_+\) such that \(t_-<\tau<t_+\) and along the vertical line over \(\bz\), no other graph \(t=f_i(\bz)\) lies between \(t_-\) and \(t_+\).  Equivalently, with the conventions \(f_0(\bz)=\alpha\) and \(f_N(\bz)=\beta\), choose
\[
        f_{r-1}(\bz)<t_-<\tau<t_+<f_{s+1}(\bz).
\]
Then \((t_-,\bz)\) belongs to precisely the layers \(E_1,\ldots,E_{r-1}\), so \(Q(t_-,\bz)=\bq_r\).  Similarly, \((t_+,\bz)\) belongs to precisely the layers \(E_1,\ldots,E_s\), so \(Q(t_+,\bz)=\bq_{s+1}\).  Since
\[
        (t_+,\bz)-(t_-,\bz)=(t_+-t_-)\be\in C,
\]
\(C\)-monotone gives
\[
        \bq_{s+1}-\bq_r=Q(t_+,\bz)-Q(t_-,\bz)\in C.
\]
Also, since \(f_r\) is \(1\)-Lipschitz, Theorem \ref{thm:rademacher} gives \(|\nabla f_r(\bz)|\le1\) at this point, and hence
\[
        (1,-\nabla f_r(\bz))\in C.
\]
Because \(\psi\ge0\) and each part belongs to $C$, every nonzero block contribution~\((\ast)\) belongs to \(\RC\).  Thus, after grouping the blocks, the integrand is \(\RC\)-valued almost everywhere and is integrable.  By Lemma \ref{lem:cone-integral}, its integral lies in \(\RC\).  This proves \(T_Q(\psi)\in\RC\).
\end{proof}

The matrix forced by the limiting argument will be the identity matrix.
Thus the final obstruction is the following elementary separation fact:
the identity matrix does not belong to the cone \(\RC\).

\begin{proposition}\label{prop:I-not-RC}
The identity matrix \(I_3\) does not belong to \(\RC\).
\end{proposition}

\begin{proof}
Let
\[
        J=\diag(1,-1,-1)
\]
and define the linear functional on \(M_3(\R)\)
\[
        \Lambda(T)=\tr(JT).
\]
For \(\bv=(v_0,\bv')\in C\) and \(\bw=(w_0,\bw')\in C\),
\[
        \Lambda(\bv\otimes \bw)
        =\bw^TJ\bv
        =v_0w_0-\bv'\cdot \bw'
        \ge v_0w_0-|\bv'||\bw'|
        \ge0.
\]
Therefore \(\Lambda(T)\ge0\) for every \(T\in\RC\).  But
\[
        \Lambda(I_3)=\tr(J)=1-1-1=-1.
\]
Thus \(I_3\notin\RC\).
\end{proof}

We are now ready to assemble the contradiction to assume that  \(\Disc\) is  an
RB-domain.
%, Lemma~\ref{lem:rb-local} would give finite-image \(C\)-isotone
%maps converging uniformly to the identity on a compact box.  Proposition
%\ref{prop:finite-image-RC} would place the corresponding tested matrices
%in \(\RC\), while the uniform convergence forces those matrices to converge
%to \(I_3\).  Since \(\RC\) is closed, this would imply \(I_3\in\RC\), contrary
%to Proposition~\ref{prop:I-not-RC}.

\begin{theorem}\label{thm:main}
Lawson's planar closed-disk domain \(\Disc\) is not an RB-domain.
\end{theorem}

\begin{proof}
Assume, towards a contradiction, that \(\Disc\) is an RB-domain.  Choose the compact rectangle, for example,
\[
        P=[-6,-4]\times[-1,1]^2\subset(-\infty,0)\times\R^2
\]
and let
\[
        \Omega=(-6,-4)\times(-1,1)^2.
\]
Take a sequence \(\varepsilon_k\rightarrow 0\).  By Lemma \ref{lem:rb-local}, there are finite-image \(C\)-monotone maps \(Q_k:P\to H\) such that
\[
        \bx-\varepsilon_k \be\le_C Q_k(\bx)\le_C \bx
        \qquad(\bx\in P)
\]
and
\[
        \sup_{\bx\in P}\|Q_k(\bx)-\bx\|\to0.
\]
Choose a non-negative \(\psi\in C_c^\infty(\Omega)\) with
\[
        \int_\Omega\psi(\bx)\dd \bx=1.
\]
Restrict $ T_{Q_k}$ on $\Omega$, by Proposition \ref{prop:finite-image-RC},
\[
        T_{Q_k}(\psi)=-\int_\Omega Q_k(\bx)\otimes\nabla\psi(\bx)\dd \bx\in\RC.
\]
% The uniform Euclidean approximation gives
% \[
% \begin{aligned}
% \left\|T_{Q_k}(\psi)+\int_\Omega \bx\otimes\nabla\psi(\bx)\dd \bx\right\|
% &\le
% \sup_{\bx\in P}\|Q_k(\bx)-\bx\|
% \int_\Omega\|\nabla\psi(\bx)\|\dd \bx.
% \end{aligned}
% \]
Hence
\[
        T_{Q_k}(\psi)
        +
        \int_\Omega \bx\otimes\nabla\psi(\bx)\,d\bx
        =
        \int_\Omega
        (\bx-Q_k(\bx))\otimes\nabla\psi(\bx)\,d\bx .
\]
Therefore, by the triangle inequality for the integral and the elementary estimate \(\|\ba\otimes\bb\|\le \|\ba\|\,\|\bb\|\), we obtain
\[
\begin{aligned}
\left\|
        T_{Q_k}(\psi)
        +
        \int_\Omega \bx\otimes\nabla\psi(\bx)\,d\bx
\right\|
&\le
        \int_\Omega
        \|(\bx-Q_k(\bx))\otimes\nabla\psi(\bx)\|\,d\bx        \\
&\le
        \int_\Omega
        \|\bx-Q_k(\bx)\|\,\|\nabla\psi(\bx)\|\,d\bx             \\
&\le
        \sup_{\by\in P}\|Q_k(\by)-\by\|
        \int_\Omega\|\nabla\psi(\bx)\|\,d\bx .
\end{aligned}
\]
Since \(\psi\in C_c^\infty(\Omega)\), the function \(\nabla\psi\) is
bounded and compactly supported.  Hence
\[
        \int_\Omega\|\nabla\psi(\bx)\|\,d\bx<\infty.
\]
Thus the right-hand side tends to \(0\) as \(k\to\infty\).
 Hence
\[
        T_{Q_k}(\psi)\to -\int_\Omega \bx\otimes\nabla\psi(\bx)\dd \bx.
\]
For the \((i,j)\)-entry of the limit, ordinary integration by parts gives
\[
        -\int_\Omega x_i\partial_j\psi(\bx)\dd \bx
        =\delta_{ij}\int_\Omega\psi(\bx)\dd \bx
        =\delta_{ij}.
\]
Therefore
\[
        -\int_\Omega \bx\otimes\nabla\psi(\bx)\dd \bx=I_3.
\]
Since \(\RC\) is closed by Proposition \ref{lem:RC-closed}, it follows that \(I_3\in\RC\), contradicting Proposition \ref{prop:I-not-RC}.  Therefore \(\Disc\) is not an RB-domain.
\end{proof}

Combining Proposition \ref{prop:jung-fs} and Theorem \ref{thm:main}, we obtain the following main result.

\begin{theorem}\label{cor:fs-not-rb}
The planar closed-disk domain is an FS-domain which is not an RB-domain.
\end{theorem}

% \begin{remark}
% The argument developed in this paper is a special case of a more general cone-domain theorem that we have proved and will present in a forthcoming sequel.  Let \(C\) be a closed, convex, pointed and generating cone in a finite-dimensional real vector space \(V\), and let
% \[
%         D_C=(-C)\cup\{\bot\}
% \]
% be the negative cone with a new least element, ordered by the cone order. Keimel~\cite{Keimel2009} claimed that these cone domains are FS-domains and asked whether they are always RB-domains.  Our general result gives the answer:
% \[
%         D_C\text{ is an RB-domain}
%         \quad\Longleftrightarrow\quad
%         C\text{ is simplicial}.
% \]
% The present paper gives the complete proof in the Lorentz-cone case which corresponds to the planar closed-disk domain. 
% \end{remark}

\bibliographystyle{plain}
\bibliography{FSvsRB}

@incollection {AbramskyJung1994,
    AUTHOR = {Abramsky, Samson and Jung, Achim},
     TITLE = {Domain theory},
 BOOKTITLE = {Handbook of logic in computer science, {V}ol.\ 3},
    SERIES = {Handb. Log. Comput. Sci.},
    VOLUME = {3},
     PAGES = {1--168},
 PUBLISHER = {Oxford Univ. Press, New York},
      YEAR = {1994},
      ISBN = {0-19-853762-X},
   MRCLASS = {68Q55 (03B70 06B35)},
  MRNUMBER = {1365749},
MRREVIEWER = {Guo-Qiang\ Zhang},
}

@Book{Gierz2003,
 Author = {Gierz, Gerhard and Hofmann, Karl and Keimel, Klaus and Lawson, Jimmie and Mislove, Michael and Scott, Dana S.},
 Title = {Continuous lattices and domains},
 FSeries = {Encyclopedia of Mathematics and Its Applications},
 Series = {Encycl. Math. Appl.},
 ISSN = {0953-4806},
 Volume = {93},
 ISBN = {0-521-80338-1},
 Year = {2003},
 Publisher = {Cambridge: Cambridge University Press},
 Language = {English},
 DOI = {10.1017/CBO9780511542725},
 zbMATH = {1849874},
 Zbl = {1088.06001}
}

@Book{Goubault2013,
 Author = {Goubault-Larrecq, Jean},
 Title = {Non-{Hausdorff} topology and domain theory. {Selected} topics in point-set topology},
 FSeries = {New Mathematical Monographs},
 Series = {New Math. Monogr.},
 Volume = {22},
 ISBN = {978-1-107-03413-6; 978-1-139-52443-8},
 Year = {2013},
 Publisher = {Cambridge: Cambridge University Press},
 Language = {English},
 DOI = {10.1017/CBO9781139524438},
 zbMATH = {6154006},
 Zbl = {1280.54002}
}

@inproceedings{Jung1990,
  title={The classification of continuous domains},
  author={Jung, Achim},
  booktitle={[1990] Proceedings. Fifth Annual IEEE Symposium on Logic in Computer Science},
  pages={35--40},
  year={1990},
  organization={IEEE}
}

@article{JungTix1998,
  title={The troublesome probabilistic powerdomain},
  author={Jung, Achim and Tix, Regina},
  journal={Electronic Notes in Theoretical Computer Science},
  volume={13},
  pages={70--91},
  year={1998},
  publisher={Elsevier}
}

@book{BoydVandenberghe,
  title={Convex optimization},
  author={Boyd, Stephen and Vandenberghe, Lieven},
  year={2004},
  publisher={Cambridge University press}
}

@book{EvansGariepy2015,
  title={Measure theory and fine properties of functions},
  author={Evans, Lawrence C},
  year={2025},
  publisher={Chapman and Hall/CRC}
}

@book{Folland1999,
  title={Real analysis: modern techniques and their applications},
  author={Folland, Gerald B},
  year={1999},
  publisher={John Wiley \& Sons}
}

@article{Keimel2009,
  title={Bicontinuous domains and some old problems in domain theory},
  author={Keimel, Klaus},
  journal={Electronic Notes in Theoretical Computer Science},
  volume={257},
  pages={35--54},
  year={2009},
  publisher={Elsevier}
}

@book{Jung1989,
  title={Cartesian closed categories of domains},
  author={Jung, Achim},
  volume={66},
  year={1989},
  publisher={Centrum voor wiskunde en informatica Amsterdam}
}

@article{lawson2008,
  title={Metric spaces and {FS}-domains},
  author={Lawson, Jimmie D},
  journal={Theoretical computer science},
  volume={405},
  number={1-2},
  pages={73},
  year={2008}
}

@article{LiangKeimel1999,
  title={Compact continuous {L}-domains},
  author={Liang, Jihua and Keimel, K},
  journal={Computers $\&$ Mathematics with Applications},
  volume={38},
  number={1},
  pages={81--89},
  year={1999},
  publisher={Elsevier}
}

@article{Plotkin1981,
  title={Post-graduate lecture notes in advanced domain theory (incorporating the "Pisa Notes")},
  author={Plotkin, Gordon D},
  journal={Dept. of Computer Science, Univ. of Edinburgh},
  year={1981}
}

@article{Rockafellar1970,
  title={Convex analysis},
  author={Tyrrell Rockafellar, R},
  journal={Princeton mathematical series},
  volume={28},
  year={1970},
  publisher={Princeton university press Princeton, NJ, USA}
}

@article{Scott1993,
  title={A type-theoretical alternative to ISWIM, CUCH, OWHY},
  author={Scott, Dana S},
  journal={Theoretical computer science},
  volume={121},
  number={1-2},
  pages={411--440},
  year={1993},
  publisher={Elsevier}
}

@article{Scott1976,
  title={Data types as lattices},
  author={Scott, Dana},
  journal={SIAM Journal on computing},
  volume={5},
  number={3},
  pages={522--587},
  year={1976},
  publisher={SIAM}
}

@inproceedings{Scott1982,
  title={Domains for denotational semantics},
  author={Scott, Dana S},
  booktitle={International Colloquium on Automata, Languages, and Programming},
  pages={577--610},
  year={1982},
  organization={Springer}
}

@article{Smyth1983,
  title={The largest cartesian closed category of domains},
  author={Smyth, Michael B},
  journal={Theoretical Computer Science},
  volume={27},
  number={1-2},
  pages={109--119},
  year={1983},
  publisher={Elsevier}
}

@article{ZouLiGuo2018,
  title={A note on the problem when {FS}-domains coincide with {RB}-domains},
  author={Zou, Zhiwei and Li, Qingguo and Guo, Lankun},
  journal={Categories and General Algebraic Structures with Applications},
  volume={8},
  number={1},
  pages={51--59},
  year={2018},
  publisher={Categories and General Algebraic Structures with Applications}
}

\end{document}